\documentclass{article}
\usepackage{amsmath}
\usepackage{amssymb}
\usepackage{amsthm}
\usepackage{comment}
\usepackage[backend=bibtex,style=numeric,citestyle=numeric]{biblatex}
\usepackage{mathtools}
\usepackage[title,toc]{appendix}
\usepackage{hyperref}
\usepackage{bbm}
\usepackage[utf8]{inputenc}
\usepackage{tikz-cd}
\usepackage{wasysym}
\usepackage{varwidth}
\usepackage{mathtools}
\usepackage[bbgreekl]{mathbbol}
\usepackage{afterpage}
\usepackage{xspace}
\usepackage{aliascnt}

\usepackage[2cell,ps]{xy}
\input xy
\UseAllTwocells
\xyoption{all}

\mathchardef\mhyphen="2D

\DeclareMathOperator{\pt}{pt}

\DeclareMathOperator{\Sch}{\mathbf{Sch}_{\mathbb Q}^f}
\newcommand{\CovSch}{\widetilde{\Sch}}

\DeclareMathOperator{\Ev}{Ev}
\newcommand{\AutSch}{\mathcal{A}}
\newcommand{\RingSch}{\mathcal{R}}

\DeclareMathOperator{\Id}{Id}

\newcommand{\CCC}{\mathcal{C}}
\newcommand{\GG}{\mathbbm{G}}

\newcommand{\UUU}{\mathcal{U}}
\newcommand{\VVV}{\mathcal{V}}

\newcommand{\FF}{\mathbbm{F}}

\newcommand{\ZZ}{\mathbbm{Z}}

\newcommand{\CC}{\mathbbm{C}}
\newcommand{\QQ}{\mathbbm{Q}}

\newcommand{\tik}{\begin{tikzcd}}
\newcommand{\tak}{\end{tikzcd}}

\makeatletter
\def\latearrow#1#2#3#4{%
  \toks@\expandafter{\tikzcd@savedpaths\path[/tikz/commutative diagrams/every arrow,#1]}%
  \global\edef\tikzcd@savedpaths{%
    \the\toks@%
    (\tikzmatrixname-#2)
    to%
    node[/tikz/commutative diagrams/every label] {$#4$}
    (\tikzmatrixname-#3)
;}}
\makeatother

\def\stik#1#2{
\begin{tikzpicture}[baseline= (a).base]%
\node[scale=#1] (a) at (0,0){%
\begin{tikzcd}[ampersand replacement=\&]%
#2
\end{tikzcd}%
};%
\end{tikzpicture}
}

\theoremstyle{plain}
\newtheorem{Theorem}{Theorem}[section]

\newtheorem{introtheorem}{Theorem}

\newaliascnt{Proposition}{Theorem}
\newtheorem{Proposition}[Proposition]{Proposition}
\aliascntresetthe{Proposition}

\newaliascnt{Corollary}{Theorem}
\newtheorem{Corollary}[Corollary]{Corollary}
\aliascntresetthe{Corollary}

\newaliascnt{Lemma}{Theorem}
\newtheorem{Lemma}[Lemma]{Lemma}
\aliascntresetthe{Lemma}

\newaliascnt{Definition}{Proposition}

\aliascntresetthe{Definition}

\newaliascnt{Notation}{Theorem}

\aliascntresetthe{Notation}

\newaliascnt{Remark}{Theorem}

\aliascntresetthe{Remark}

\newcommand{\Rami}[1]{{{#1}}}

\usepackage[utf8]{inputenc}

\title{Autoequivalences of the category of schemes}
\author{Avraham Aizenbud and Adam Gal}

\begin{document}

\maketitle
\begin{abstract}
We prove that there is no non-trivial aoutoequvivalence of the category of schemes of finite type over $\QQ$.
\end{abstract}

\tableofcontents

\section{Introduction}
In this paper we prove the following theorem
\begin{introtheorem}\label{thm:main}
Let  $\Sch$ be the category of schemes of finite type over $\QQ$. Then any autoequivalence of this category is isomorphic to the identity functor by a unique isomorphism.
\end{introtheorem}
\subsection{Motivation}
The notions of schemes and algebraic varieties have several definitions (see e.g. \cite{EGA1}, \cite{Jantzen}, \cite{Kempf}). When one proves that two of those notions are equivalent one should ask what does this mean. Is it enough to find an equivalence between the corresponding categories? Should this equivalence preserve some structure (for example the category of coherent sheaves or some classes of maps, like flat or proper)? Should it be canonical in some sense?

\autoref{thm:main} means that if one is interested in the category $\Sch$ one can just check equivalence of categories. This equivalence will be unique and thus any other structure that this category possesses can be reconstructed from the categorical structure.

\subsection{related results} 
There are several works that study auto-equvivalences of categories. In \cite{Mak} several methods to prove that there are no non-trivial auto-equvivalences of a given category are listed. Those methods are implemented in several cases including the categories of groups and sets. Another related example can be found in \cite[Theorem 5.2.1]{EHS}. There are interesting cases when the category of autoequivalences is non-trivial (see e.g. \cite{Mukai}, \cite[Theorem 2.1.4]{Orlov})

\subsection{Idea of the proof}
The main part of the proof is \autoref{sec:ga.pres} where we prove that any autoequivalence of  $\Sch$ preserves the affine line as a ring object (See \autoref{thm:Ga}). We deduce it from the fact that the affine line pointed by $1$ is an initial object in the category of torsion free, pointed, abelian algebraic groups (See \autoref{prop:init.ob}).

The deduction of \autoref{thm:main} is based on the following facts
\begin{itemize}
    \item Any morphism of affine spaces can be expressed using the ring structure on $\mathbb A^1$
    \item Any affine scheme is a fiber product of affine spaces.
    \item Any quasicompact separated scheme  is a pushout of affine schemes.
    \item Any quasicompact scheme is a pushout of separated schemes.
\end{itemize}

\subsection{Possible extensions and limitations}
We believe that \autoref{thm:main} can be generalized to other categories of schemes. In fact the proof in this paper also holds for two subcategories of  $\Sch$:  (affine schemes and separated schemes). 

One can try to generalize the  argument in the paper beyond finite type schemes and to replace $\QQ$ by a general prime field or $\ZZ$. The main obstacle will be \autoref{prop:init.ob}, which generalization will require an understanding of the theory of group-schemes in wider generality\footnote{Also in order to make it work over $\FF_2$ one should modify the definition of $\CC$.}.   

On the other hand, \autoref{thm:main} stops being valid when we replace $\QQ$ by an arbitrary ring, or more generally a scheme $S$. This is because any automorphism of $S$ gives an autoequivalence of $\Sch_S^f$. However, it is concievable that there are no other  autoequivalence of $\Sch_S^f$. 
\subsection{Structure of the paper}

In \autoref{sec:red.2.ga} we prove proposition \autoref{Prop:fullyfaithful} that reduces the main result to \autoref{thm:Ga}.

In \autoref{sec:ga.pres} we prove \autoref{thm:Ga}.

\subsection*{Acknowledgements}
We \Rami{t}hank Inna Entova-Aizenbud and Vladimir Hinich for useful discussions.
\section{Reduction to \texorpdfstring{$\GG_a$}{Ga}}
\label{sec:red.2.ga}

\begin{Lemma}
Let $\Phi:\Sch\to\Sch$ be an autoequivalence, and suppose that $\Phi(\GG_a)\cong\GG_a$ as ring objects. Then $\Phi\cong\Id$.
\end{Lemma}

More precisely, we will prove the following

\begin{Proposition}
\label{Prop:fullyfaithful}
Consider the categories $\AutSch$ - of endofunctors of $\Sch$ which preserve all finite limits and colimits - and $\RingSch$ - of all unital ring objects in $\Sch$. Consider the functor $\Ev:\AutSch\to\RingSch$ given by evaluation at $\GG_a$ (with the standard ring structure). Then $\Ev$ is fully faithful.
\end{Proposition}
Before proving the proposition note that the Proposition implies the Lemma since the isomorphism $\Phi(\GG_a)\cong\GG_a=\Id(\GG_a)$ then lifts to an isomorphism $\Phi\cong\Id$.

We also note the following 

\begin{Corollary}
The only endomorphism of $\Id:\Sch\to\Sch$ is the trivial one.
\end{Corollary}

\begin{proof}
The corollary follows from the theorem since 
the only morphism of unital ring schemes from $\GG_a$ to itself is trivial.
\end{proof}

\begin{proof}[Proof of \autoref{Prop:fullyfaithful}]
Let $F,G\in \AutSch$ and $\phi:F\to G$. It is obvious that $\Ev(\phi):F(\GG_a)\to(\GG_a)$ is a morphism of unital ring schemes.

\textbf{Faithfulness:} 

Let $I:D\to \Sch$ be a finite diagram and $X$ a limit (or a colimit) of $I$. A natural transformation gives a map $F\circ I \to G\circ I$. By assumption $G(X)$ is a limit of $G\circ I$, so there is a unique morphism $F(X)\to G(X)$ which makes all the relevant diagrams commute. This means that the natural transformation is determined on $X$ by its values on the image of $I$. The same is obviously true when $X$ is a colimit of $I$.

It is known that any scheme can be obtained from $\GG_a$ via a sequence of finite limits and colimits, so $\Ev(\phi)$ determines $\phi$.

\textbf{Fullness:}
We have to extend any morphism of ring object $\psi:F(\GG_a)\to(\GG_a)$ to a morphism $:F\to G$

\newcounter{Step}
\textbf{Step \theStep} \stepcounter{Step}
It is obvious that $\psi$ extends to the subcategory spanned by powers of $\GG_a$ (including the point). 

\textbf{Step \theStep} \stepcounter{Step}
We can extend $\psi$ to affine varieties:

Any affine variety $X$ can be written as a pullback 
\[
\stik{1}{
X \ar{d} \ar[phantom]{dr}[above,yshift=-0.3pc,xshift=-0.8pc,scale=1.5]{\ulcorner} \ar[hook]{r} \& \GG_a^n \ar{d}\\
\pt \ar[hook]{r} \& \GG_a^k
}
\]
So this gives a candidate morphism $F(X)\to G(X)$ as $G$ preserves pullbacks. If we show that this is natural in $X$ this will in particular show that the morphism does not depend on a choice of presentation. So consider another affine variety $Y$ and a morphism $X\xrightarrow{f} Y$. Write $Y$ as a pullback
\[
\stik{1}{
Y \ar{d} \ar[phantom]{dr}[above,yshift=-0.3pc,xshift=-0.8pc,scale=1.5]{\ulcorner} \ar[hook]{r} \& \GG_a^m \ar{d}\\
\pt \ar[hook]{r} \& \GG_a^r
}
\]
By basic properties of affine varieties, the morphism $f$ can be extended to a commutative square

\[
\stik{1}{
X \ar[hook]{d} \ar{r}{f} \& Y \ar[hook]{d}\\
\GG_a^n \ar{r} \& \GG_a^m
}
\]
Adding in our candidate morphisms, and noting that $F,G$ preserve monomorphisms, we get a cube

\[\stik{1}{
{} \& G(X) \ar{rr}{G(f)} \ar[hook]{dd}[left,yshift=-1pc]{} \& {} \& G(Y) \ar[hook]{dd}{}\\
F(X) \ar[crossing over]{rr}[xshift=1pc]{F(f)} \ar[hook]{dd}[left]{} \ar{ur}{} \& {} \& F(Y) \ar[hook]{dd}[yshift=-1pc,left,xshift=-0.5pc]{} \ar{ur}{} \& {}\\
{} \& G(\GG_a^n) \ar{rr}[xshift=-1pc]{} \& {} \& G(\GG_a^m)\\
F(\GG_a^n) \ar{rr}{} \ar{ur}{} \& {} \& F(\GG_a^m) \ar{ur}
\latearrow{commutative diagrams/crossing over,commutative diagrams/hook}{2-3}{4-3}{}
}
\]
with all faces commuting except (possibly) the top. The sides commute because of the way we defined our candidate morphisms, the front and back are images of commutative squares, and the bottom commutes because it involves only powers of $\GG_a$. 

\begin{Lemma}
\label{cubelemma}
Suppose we are given a cube in a category $\CCC$ with all faces commuting except the possibly the top, and all vertical maps monomorphic. Then the top face commutes as well.
\end{Lemma}

\begin{proof}
From the definition of a monomorphism it is enough to show that the morphisms of the top face post-composed with a vertical are equal. This follows from commutativity of the other faces.
\end{proof}
The commutativity of the top face is the naturality of our candidate.
\textbf{Step \theStep} \stepcounter{Step}
We can extend $\psi$ to the category of separated schemes:

Consider the category $\CovSch_{s}$ whose objects are a separated scheme $X$ along with an affine cover $\UUU\to X$. The morphisms are commutative squares
\[
\stik{1}{
\VVV \ar{r} \ar[two heads]{d} \& \UUU \ar[two heads]{d} \\
Y \ar{r} \& X
}
\]
The functors $F,G$ can be extended to autofunctors $\widetilde{F},\widetilde{G}$ of $\CovSch_{s}$ in the obvious way. Moreover, $\psi$ can be extended to a morphism $\widetilde{\psi}:\widetilde{F}\to\widetilde{G}$ as follows:

Consider a covered scheme $\UUU\to X$. The cover displays $X$ as a colimit of the diagram $\UUU\times_X\UUU\to \UUU$\Rami{.}
Since $X$ is separated the intersections are all afine and thus $\UUU\times_X\UUU\to \UUU$ is affine. By the previous step we have a map $F(\UUU)\to G(\UUU)$ and so this induces a unique map $F(X)\to G(X)$ which is our candidate for $\psi_{{}_{\UUU\to X}}$.

To see that this is a natural map, we use a similar argument to the previous step. Consider a morphism of covered schemes
\[
\stik{1}{
\VVV \ar{r} \ar[two heads]{d} \& \UUU \ar[two heads]{d} \\
Z \ar{r} \& X
}
\]

From the above we get a cube
\[
\stik{1}{
{} \& G(\VVV) \ar{rr}{} \ar[two heads]{dd}[left,yshift=-1pc]{} \& {} \& G(\UUU) \ar[two heads]{dd}{}\\
F(\VVV) \ar[crossing over]{rr}[xshift=1pc]{} \ar[two heads]{dd}[left]{} \ar{ur}{} \& {} \& F(\UUU) \ar[two heads]{dd}[yshift=-1pc,left,xshift=-0.5pc]{} \ar{ur}{} \& {}\\
{} \& G(Z) \ar{rr}[xshift=-1pc]{} \& {} \& G(X)\\
F(Z) \ar{rr}{} \ar{ur}{} \& {} \& F(X) \ar{ur}
\latearrow{commutative diagrams/crossing over,commutative diagrams/two heads}{2-3}{4-3}{}
}
\]
with all sides commuting except possibly the bottom, and all verticals epimorphisms. The dual of \autoref{cubelemma} implies that the bottom commutes as well, which is what we want.

To finish the step we need to show that $\psi_{{}_{\UUU\to X}}$ doesn't depend on the cover, i.e. that for any two covers $\UUU,\UUU'$ of $X$ we have $\psi_{{}_{\UUU\to X}}=\psi_{{}_{\UUU'\to X}}$. If we had a map of the covers over $\Id_X$ this would be immediate. This is not always the case, but we can always take the pullback of the covers to get a common refinement $\UUU\leftarrow\UUU''\to\UUU\Rami{'}$. Hence, $\psi_{{}_{\UUU\to X}}=\psi_{{}_{\UUU''\to X}}=\psi_{{}_{\UUU'\to X}}$.

\textbf{Step \theStep} \stepcounter{Step}
We can extend to all schemes:

This follows in exactly the same way as the previous step, using the fact that every scheme has a cover by separated schemes, with separated intersections.

\end{proof}
\section{Proof of main theorem}\label{sec:ga.pres}
In order to finish the proof of the main theorem one has to prove
\begin{Theorem}\label{thm:Ga}
Let $\Phi$ be an autoequivalence of $\Sch$. Then $\Phi(\GG_a)\simeq\GG_a$ as ring objects.
\end{Theorem}

What we need to show then, is that $\GG_a$ can be characterized up to isomorphism in categorical terms.

Let $G$ be a commutative group object in $\Sch$. We define the map \\ $m_2:G\rightarrow G$ as the composition of the diagonal $G\rightarrow G\times G$ and the product $G\times G\rightarrow G$. This is a group endomorphism of $G$, since $G$ is commutative. Note that $m_2$ is defined entirely in categorical terms (in the category of commutative group objects).

Let $\CCC$ be the category consisting of pointed commutative group objects in $\Sch$ such that $\ker(m_2)$ is trivial. By ``pointed" we mean with a specified $\QQ$-point.

\begin{Proposition}\label{prop:init.ob}
$(\GG_a,1)$ is an initial object in $\CCC$.
\end{Proposition}

For the proof we will need a lemma:

\begin{Lemma}
\label{lem:affineroot}
For any affine commutative group scheme, the map $m_2$ defined above is onto.
\end{Lemma}
Before proving the lemma we will deduce the proposition from it.
\begin{proof}[Proof of \autoref{prop:init.ob}]
First, we show that $\GG_a$ is in $\CCC$. If we identify $\QQ[\GG_a]$ with $\QQ[t]$, then the map $m_2$ corresponds to the map $t\mapsto 2\cdot t$. This map is onto and hence the kernel of $m_2$ is trivial.

Now let $(X,b)\in \CCC$. Suppose we have a map $\GG_a\rightarrow X$ which sends $1$ to $b$, then this map is uniquely defined by this requirement because the multiples of $1$ are a dense subset in $\GG_a$. That is, if we consider the system of maps $n:\pt\to \GG_a$ then any two maps $f,g:\GG_a\to X$ such that $f\circ n=g\circ n$ for all $n$ must coincide.

Existence: We will show that the condition that $\ker(m_2)$ is trivial implies that $X$ is a commutative unipotent group, hence a product of copies of $\GG_a$.

It is known (see e.g. \cite[Chapter 10]{Milne_ag}) that $X$ admits an exact sequence 
\[
1\rightarrow B \rightarrow X \rightarrow A \rightarrow 1
\]
where $A$ is an abelian variety, and $B$ is an affine commutative group. We claim that the condition that $\ker(m_2)$ is trivial implies that $A$ is trivial. Indeed if it is not then we have a 2-torsion element $a\in A(\CC)$ (see e.g. \cite{Milne_ab}). Let $x$ be a preimage of $a$ in $X$. By assumption $x^2$ will be in the kernel, which is isomorphic to $B$. By
 \autoref{lem:affineroot} we can extract a root in $B(\CC)$ for $x^2$ and get a nontrivial 2-torsion element in $X(\CC)$, which contradicts the fact that $\ker(m_2)$ is trivial.

Assuming \autoref{lem:affineroot}, we are reduced to the case when $X$ is affine, and hence $X=X_s\times X_u$ where $X_s$ is a torus and $X_u$ is commutative unipotent.

The condition that $\ker(m_2)$ is trivial implies once again that $X_s$ is trivial and we are finished.
\end{proof}

\begin{proof}[Proof of \autoref{lem:affineroot}]
Any affine commutative group is a product of a commutative unipotent group - i.e. a vector space - and a torus. In both cases the lemma obviously holds in characteristic not 2.
\end{proof}

\begin{proof}[Proof of \autoref{thm:Ga}]
From \autoref{prop:init.ob} it follows that $\Phi(\GG_a)$ is isomorphic to $\GG_a$ as a pointed group object. Denote the isomorphism by $\alpha$.

We need to show that $\alpha$ preserves the ring structure. 
Consider any integer number $n$ as an element in $\GG_a(\QQ)$ in the standard way. Let $e \in \Phi(\GG_a)(\QQ)$ be the unit w.r.t. multiplication. Let  $ne:=\underbrace{e+\cdots +e}_{n \text{ times}}$.  
We know that $\alpha(1)= e$ then $\alpha(n)=ne$ for any $n\in\ZZ$. Thus $\alpha(nm)=(nm)e=(ne)(me)=\alpha(n)\alpha(m)$. This means that the diagram
\[
\stik{1}{
\Phi(\GG_a)\times\Phi(\GG_a)\ar{r}{\alpha\times\alpha} \ar{d}{m}\& \GG_a\times \GG_a \ar{d}{m} \\
\Phi(\GG_a) \ar{r}{\alpha} \& \GG_a
}
\]
commutes on a dense subset of $\Phi(\GG_a)\times\Phi(\GG_a)$ (in the sense discussed in the proof of \autoref{prop:init.ob}) and hence commutes. Therefore $\alpha$ is an isomorphism of ring objects.
This completes the proof.
\end{proof}

\printbibliography

\end{document}